\documentclass[twoside,a4paper]{amsart}

\usepackage{xcolor}
\definecolor{webgreen}{rgb}{0,.5,0}
\definecolor{webbrown}{rgb}{.6,0,0} 
\definecolor{RoyalBlue}{cmyk}{1, 0.50, 0, 0}
\usepackage[colorlinks=true, breaklinks=true, urlcolor=webbrown, linkcolor=RoyalBlue, citecolor=webgreen,  backref=page]{hyperref}
\usepackage{epsfig, graphicx, subfigure}
\usepackage{verbatim, setspace,pifont}
\usepackage{amsmath, amssymb}
\usepackage{mathptmx}

\newcommand{\R}{{\mathbb R}}

\newcommand{\D}{{\mathbb D}} 

\newcommand{\ic}{\mathrm{i}}
\newcommand{\dd}{\mathrm{d}}

\newtheorem{theorem}{Theorem}
\newtheorem{lemma}{Lemma}

\begin{document}

\title[Expected number of real zeros of Kac-Geronimus polynomials]{An asymptotic expansion for the expected number of real zeros of Kac-Geronimus polynomials}

\author{Hanan Aljubran}

\author{Maxim L. Yattselev}

\address{Department of Mathematical Sciences, Indiana University-Purdue University Indianapolis, 402~North Blackford Street, Indianapolis, IN 46202}

\email{\href{mailto:haljubra@iu.edu}{haljubra@iu.edu} (Hanan Aljubran)}

\email{\href{mailto:maxyatts@iupui.edu}{maxyatts@iupui.edu} (Maxim Yattselev)}

\thanks{The work of the first author was done towards completion of her Ph.D. degree at Indiana University-Purdue University Indianapolis under the direction of the second author. The research of the second author was supported in part by grants from the Simons Foundation, CGM-354538 and CGM-706591. Both authors are grateful for hospitality of American Institute of Mathematics where a part of this project was carried out during workshop ``Zeros of random polynomials''.}



\begin{abstract}
Let \( \{\varphi_i(z;\alpha)\}_{i=0}^\infty \), corresponding to \( \alpha\in(-1,1) \), be orthonormal Geronimus polynomials.  We study asymptotic behavior of the expected number of real zeros, say \( \mathbb E_n(\alpha) \), of random polynomials
\[
P_n(z) := \sum_{i=0}^n\eta_i\varphi_i(z;\alpha),
\]
where \( \eta_0,\dots,\eta_n \) are i.i.d. standard Gaussian random variables.  When \( \alpha=0 \), \( \varphi_i(z;0)=z^i \) and \( P_n(z) \) are called Kac polynomials. In this case it was shown by Wilkins that \( \mathbb E_n(0) \) admits an asymptotic expansion of the form
\[
\mathbb E_n(0) \sim \frac2\pi\log(n+1) + \sum_{p=0}^\infty A_p(n+1)^{-p}
\]
(Kac himself obtained the leading term of this expansion). In this work we obtain a similar expansion of \( \mathbb E(\alpha) \) for \( \alpha\neq 0 \). As it turns out, the leading term of the asymptotics in this case is \( (1/\pi)\log(n+1) \).
\end{abstract}

\maketitle

\section{Introduction and Main Results}

Random polynomials is a relatively old subject with initial contributions by Bloch and P\'olya, Littlewood and Offord, Erd\"os and Offord, Arnold, Kac, and many other authors. An interested reader can find a well referenced early history of the subject in the books by Bharucha-Reid and Sambandham \cite{Bharucha-ReidSambandham}, and by Farahmand \cite{Farahmand}. In \cite{Kac43}, Kac considered random polynomials
\begin{equation}
\label{kac}
P_n(z) = \eta_0 + \eta_1 z + \cdots +\eta_nz^n,
\end{equation}
where \( \eta_i \) are i.i.d. standard real Gaussian random variables. He has shown that \( \mathbb E_n(\Omega) \), the expected number of zeros of \( P_n(z) \) on a measurable set \( \Omega\subset\R \), is equal to
\begin{equation}
\label{Kac}
\mathbb E_n(\Omega) = \frac1\pi\int_\Omega\frac{\sqrt{1-h_{n+1}^2(x)}}{|1-x^2|}\dd x, \quad h_{n+1}(x) = \frac{(n+1)x^n(1-x^2)}{1-x^{2n+2}},
\end{equation}
from which he proceeded with an asymptotic formula
\begin{equation}
\label{KacE}
\mathbb E_n(\R) = \frac{2+o(1)}\pi\log(n+1) \quad \text{as} \quad n\to\infty.
\end{equation}
It was shown by Wilkins \cite{Wilk88}, after some intermediate results cited in \cite{Wilk88}, that there exist constants \( A_p \), \( p\geq 0 \), such that \( \mathbb E_n(\R) \) has an asymptotic expansion of the form
\begin{equation}
\label{Wilkins}
\mathbb E_n(\R) \sim \frac2\pi\log (n+1) + \sum_{p=0}^\infty A_p(n+1)^{-p},
\end{equation}
where
\begin{equation}
\label{A0}
A_0 = \frac2\pi\left(\log2+\int_0^1\frac{f(t)}t\dd t + \int_1^\infty \frac{f(t)-1}t \dd t\right), \quad f(t) := \sqrt{1-\left(\frac{2t}{e^t-e^{-t}}\right)^2}.
\end{equation}

Many subsequent results on random polynomials are concerned with relaxing the conditions on random coefficients, see, for example, \cite{IbMas71,NgNgVu16,DNgVu18}, or the behavior of the counting measures of zeros of random polynomials as in \cite{MR1935565, MR2318650, MR3262481, MR3306686,  MR3489554, MR3571446, MR3739138, LubPritXie18, BlDau19, Dau}. Our primary interest lies in studying the expected number of real zeros when the basis is a family of orthogonal polynomials in the spirit of \cite{MR0268933,MR653589,MR1377012,LubPritXie16}. More precisely, Edelman and Kostlan \cite{EdKos95} considered random functions of the form
\begin{equation}
\label{Vanderbei}
P_n(z) = \eta_0f_0(z) + \eta_1 f_1(z) + \cdots +\eta_nf_n(z),
\end{equation}
where \( \eta_i \) are certain real random variables and \( f_m(z) \) are arbitrary functions on the complex plane that are real on the real line. Using a beautiful and simple geometrical argument they have shown\footnote{In fact, Edelman and Kostlan derive an expression for the real intensity function for any random vector \( (\eta_0,\ldots,\eta_n) \) in terms of its joint probability density function and of \( v(x) \).} that if \( \eta_0,\ldots,\eta_n \) are elements of a multivariate real normal distribution with mean zero and covariance matrix \( C \) and the functions \( f_m(z) \) are differentiable on the real line, then
\[
\mathbb E_n(\Omega) = \int_\Omega\rho_n(x)\dd x, \quad \rho_n(x) = \left.\frac1\pi\frac{\partial^2}{\partial s\partial t}\log\left( v(s)^\mathsf{T}Cv(t)\right)\right|_{t=s=x},
\]
where \( v(x) = \big(f_0(x),\ldots,f_n(x)\big)^\mathsf{T} \). If random variables \( \eta_i \) in \eqref{Vanderbei} are again i.i.d. standard real Gaussians, then the above expression for \( \rho_n(x) \) specializes to
\begin{equation}
\label{real-intensity}
\rho_n(x) = \frac1\pi\frac{\sqrt{K_{n+1}(x,x)K_{n+1}^{(1,1)}(x,x)-K_{n+1}^{(1,0)}(x,x)^2}}{K_{n+1}(x,x)}
\end{equation}
(this formula was also independently rederived in \cite[Proposition~1.1]{LubPritXie16} and \cite[Theorem~1.2]{uVan}), where \( K_{n+1}(x,y) := K_{n+1}^{(0,0)}(x,y) \) and
\[
K_{n+1}^{(l,k)}(x,y) := \sum_{i=0}^nf_i^{(l)}(x)\overline{f_i^{(k)}(y)}.
\]

We are interested in the case where the spanning functions in \eqref{Vanderbei} are taken to be orthonormal polynomials on the unit circle. Recall \cite[Theorem~1.5.2]{Simon1} that monic orthogonal polynomials, say \( \Phi_m(z) \), satisfy the recurrence relations
\begin{equation}
\label{3term}
\left\{
\begin{array}{l}
\displaystyle \Phi_{m+1}(z) = z\Phi_m(z) - \overline \alpha_m\Phi_m^*(z), \medskip \\
\displaystyle \Phi_{m+1}^*(z) = \Phi_m^*(z) - \alpha_mz\Phi_m(z),
\end{array}
\right.
\end{equation}
where the recurrence coefficients \( \{\alpha_m\} \) belong to the unit disk \( \mathbb{D} \) and are uniquely determined by the measure of orthogonality. Furthermore, the orthonormal polynomials, which we denote by \( \varphi_m(z) \), are given by
\begin{equation}
\label{ortho-norm}
\varphi_m(z) = \rho_m^{-1}\Phi_m(z), \quad \rho_m := \prod_{i=0}^{m-1}\sqrt{1-|\alpha_i|^2}.
\end{equation}
Since the functions \( f_m(z) \) in \eqref{Vanderbei} must be real-valued on the real line, we are only interested in real recurrence coefficients, i.e., \( \alpha_m\in(-1,1) \) for all \( m\geq0 \). It is known \cite{Yatt} that when \( m^p|\alpha_m| \) is a bounded sequence for some \( p>3/2 \), estimate \eqref{KacE} remains valid for random polynomials \eqref{Vanderbei} with \( f_m(z) = \varphi_m(z) \) given by \eqref{3term}--\eqref{ortho-norm}. Moreover, if the recurrence coefficients decay exponentially, it was shown by the authors in \cite{AlYa19} that the expected number of real zeros has a full asymptotic expansion of the form \eqref{Wilkins} with the constant term still given by \eqref{A0}.

The previous works suggest that the constant \( \pi/2 \) in front of \( \log(n+1) \) in \eqref{KacE} and \eqref{Wilkins} might change if the recurrence coefficients decay slowly or do not decay at all. In this note we support this guess by considering random polynomials of the form
\begin{equation}
\label{random-opuc}
P_n(z) = \eta_0\varphi_0(z;\alpha) + \eta_1 \varphi_1(z;\alpha) + \cdots +\eta_n\varphi_n(z;\alpha),
\end{equation}
which we call Kac-Geronimus polynomials, where \( \eta_i \) are i.i.d. standard real Gaussian random variables and
\begin{equation}
\label{Geron}
\varphi_m(z;\alpha) = \rho^{-m}\Phi_m(z;\alpha) , \quad \rho:= \sqrt{1-\alpha^2},
\end{equation}
are real Geronimus polynomials, that is, polynomials \( \Phi_m(z;\alpha) \) satisfying \eqref{3term} with \( \alpha_m=\alpha\in(-1,1) \) for all \( m\geq0 \). The measure of orthogonality for general Geronimus polynomials, i.e., \( \alpha_m=\alpha\in\D \), is explicitly known, see \cite[Section~1.6]{Simon1}, and is supported by
\[
\Delta_\alpha := \big\{e^{i \theta } : 2 \arcsin(|\alpha|) \leq \theta \leq 2 \pi - 2 \arcsin(|\alpha|) \big\}
\]
with a possible pure mass point, which is present if and only if \( |\alpha+1/2|>1/2 \). When \( \alpha=0 \), one can clearly see from \eqref{3term} that \( \Phi_m(z;0) = z^m \) and therefore Kac-Geronimus polynomials \eqref{random-opuc} specialize to Kac polynomials \eqref{kac}.
 
For random polynomials \eqref{Vanderbei} with \( f_m(z) = \varphi_m(z) \) given by \eqref{3term}--\eqref{ortho-norm} it can be easily shown using the Christoffel-Darboux formula, see \cite[Theorem~1.1]{Yatt}, that \eqref{real-intensity} can be rewritten as
\begin{equation}
\label{set-up}
\rho_n(x) = \frac1\pi\frac{\sqrt{1-h_{n+1}^2(x)}}{|1-x^2|}, \quad h_{n+1}(x) := \frac{(1-x^2)b_{n+1}^\prime(x)}{1-b_{n+1}^2(x)}, \quad b_{n+1}(x) := \frac{\varphi_{n+1}(x)}{\varphi_{n+1}^*(x)},
\end{equation}
where \( \varphi_{n+1}^*(x):=x^{n+1}\varphi_{n+1}(1/x) \) is the reciprocal polynomial (there is no need for conjugation as all the coefficients are real). 

\begin{theorem}
\label{thm:0}
Let \( P_n(z) \) be given by \eqref{random-opuc}--\eqref{Geron} with \( \alpha\in(-1,0)\cup(0,1) \). Define
\begin{equation}
\label{r}
r(z) := \sqrt{(z-1)^2+4\alpha^2z}
\end{equation}
to be the branch holomorphic in \( \mathbb{C} \backslash \Delta_\alpha \) such that \( r(z)/z \to 1\) as \( z\to \infty \). Then it holds that
\begin{equation}
\label{bn-asymp}
\lim_{n\to\infty} b_{n+1}(z) = \frac{-2\alpha}{r(z)+1-z}
\end{equation}
locally uniformly in \( \D \). Moreover, it holds that
\begin{equation}
\label{hn-asymp}
h_{n+1}(x) = -\alpha\frac{x+1}{r(x)}\left( 1 + \mathcal O\left((1-x)^2 (n+1) e^{-\sqrt{n+1}/\rho}\right) \right),
\end{equation}
for \( -1+(n+1)^{-1/2}\leq x \leq 1-\delta_\alpha^{n+1} \), where \( \mathcal O(\cdot) \) does not depend on \( n \) and \( \delta_\alpha := 0 \) when \( \alpha<0 \) while \( \delta_\alpha := ((1-\alpha)/(1+\alpha))^{1/3} \) when \( \alpha>0 \).
\end{theorem}

Observe that \(b_{n+1}(1) = h_{n+1}(1) = 1 \) for all \( n \) and these equalities remain true in the limit when \( \alpha<0 \). However, \( b(1) = h(1) = -1 \) when \( \alpha> 0 \) . This change is due to a single zero of \( \varphi_m(z;\alpha) \) that approaches \( 1 \) as \( m\to\infty \) for every fixed \( \alpha>0 \), see Figure~\ref{fig:1}, and is the reason we need to introduce \( \delta_\alpha \) in \eqref{hn-asymp}.

\begin{figure}[ht!]
\begin{center}
\subfigure[\( b_4(x)\) and \( b(x) \)]{\includegraphics[scale=.3]{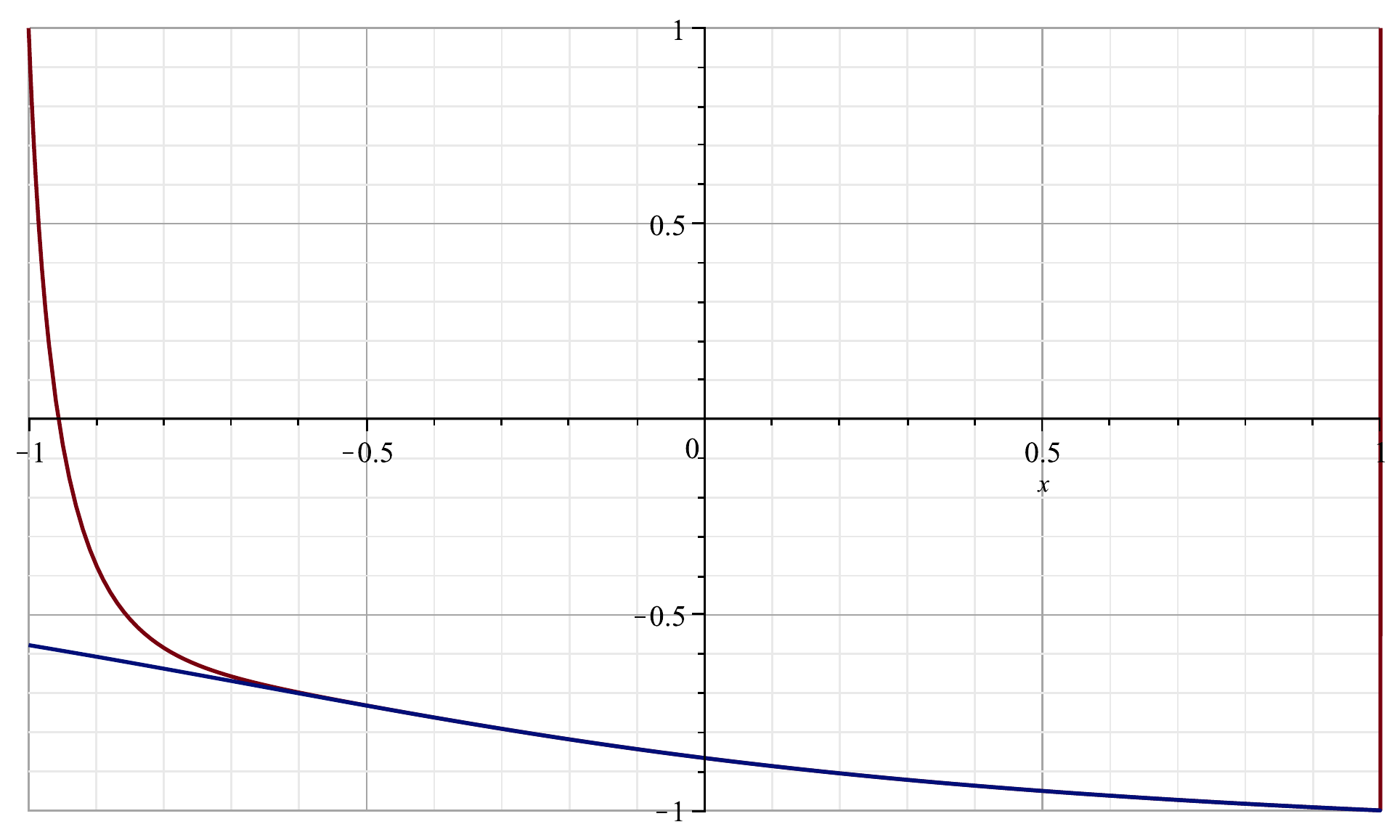}
}\quad
\subfigure[\( h_4(x)\) and \( h(x) \)]{\includegraphics[scale=.3]{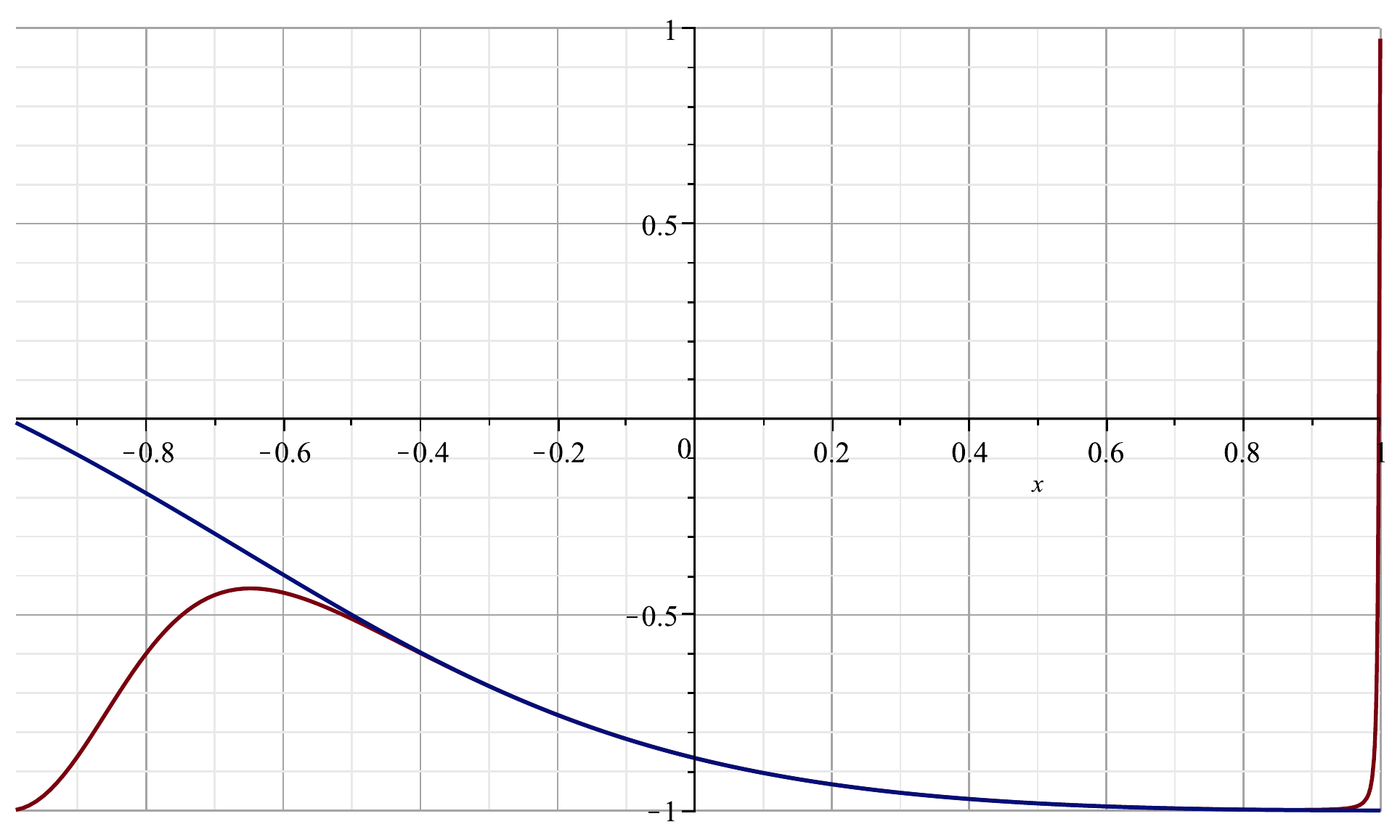}
}
\end{center}
\caption{\small The graphs of \( b_4(x)\) and \( b(x) \) (panel (a)) and \( h_4(x)\) and \( h(x) \) (panel (b)) on \( [-1,1] \) for \( \alpha=\sqrt3/2 \).}
\label{fig:1}
\end{figure} 

Let \( \mathbb E_n(\alpha) \) be the expected number of real zeros of random polynomials \eqref{random-opuc}--\eqref{Geron}. It is easy to see that \( b_m(1/x) = 1/b_m(x) \) and therefore \( b_m^\prime(1/x) = x^2b_m^\prime(x)/b_m^2(x) \). Thus, we get from \eqref{set-up} that \( h_m(1/x) = h_m(x) \) and therefore
\begin{equation}
\label{Enalpha}
\mathbb E_n(\alpha) = \frac2\pi\int_{-1}^1\frac{\sqrt{1-h_{n+1}^2(x)}}{1-x^2}\dd x.
\end{equation}
Using this formula we can prove the following theorem that constitutes the main result of this work.

\begin{theorem}
\label{thm:1}
Let \( P_n(z) \) be random polynomials given by \eqref{random-opuc}--\eqref{Geron} with \( \alpha\in(-1,0)\cup(0,1) \). Then there exist constants  \( A_p^{\alpha,(-1)^n} \), \( p\geq1 \), that do depend on the parity of \( n \), such that \( \mathbb E_n(\alpha) \), the expected number of real zeros of \( P_n(z) \), satisfies
\[
\mathbb E_n(\alpha) = \frac1\pi\log(n+1)+A_0^\alpha + \sum_{p=1}^{N-1}A_p^{\alpha,(-1)^n}(n+1)^{-p} + \mathcal O_N\left((n+1)^{-N}\right) 
\]
for any integer \( N \), all \( n \) large, where \( \mathcal O_N(\cdot) \) depends on \( N \), but is independent of \( n \), and
\[
A_0^\alpha = \frac{A_0+1+\mathrm{sgn}(\alpha)}2 + \frac1\pi\log\frac2{|\alpha|}
\]
with \( A_0 \) given by \eqref{A0} and \( \mathrm{sgn}(\alpha):=\alpha/|\alpha| \).
\end{theorem}

Notice that \( A_0^{|\alpha|} = A_0^{-|\alpha|} + 1\). This is due to the fact that polynomials \( \varphi_m(x;|\alpha|) \) have a zero exponentially close to \( 1 \) while polynomials \( \varphi_m(x;-|\alpha|) \) do not have such a zero.

\section{Proof of Theorem~\ref{thm:0}}

\begin{lemma}
\label{lem:1}
It holds that
\begin{equation}
\label{b-value}
b_{n+1}(z) = \frac{ \phi(z) - 2(1+\alpha) - \epsilon^{n+1}(z) (\psi(z) - 2(1+\alpha)) }{ \phi(z) - 2(1+\alpha)z - \epsilon^{n+1}(z) (\psi(z) - 2(1+\alpha)z)  }
\end{equation}
where  \(\phi(z) := z+1 + r(z) \), \( \psi(z) := z+1 - r(z) \), \( \epsilon(z) := \psi(z)/\phi(z) \), and \( r(z) \) was defined in \eqref{r}. In particular, \eqref{bn-asymp} takes place.
\end{lemma}
\begin{proof}
Let \( U_m(y) \) be the degree \( m \) Chebysh\"ev polynomial of the second kind, that is,
\[
U_m(y) = \frac{\left( y + \sqrt{y^2-1}  \right)^{m+1} - \left( y - \sqrt{y^2-1} \right)^{m+1}}{2 \sqrt{y^2-1} },
\]
where for definiteness we take the branch \( \sqrt{y^2-1} = y + \mathcal O(1) \) as \( y\to\infty \) with the cut along \( [-1,1] \). It has been shown in \cite[Theorem~3.1]{Simanek18} that 
\begin{equation}
\label{varphi-asymp}
\left\{
\begin{array}{l}
\displaystyle \varphi_{m}(z;\alpha) = z^{m/2} \left( U_m \left( \frac{z+1}{2\rho \sqrt{z}} \right) - \frac{1+\overline\alpha}{\rho \sqrt{z}} U_{m-1}\left( \frac{z+1}{2\rho \sqrt{z}} \right)  \right), \medskip \\
\displaystyle  \varphi_m^*(z;\alpha) = z^{m/2} \left( U_m\left( \frac{z+1}{2\rho \sqrt{z}} \right) - \frac{\sqrt{z}(1+\alpha)}{\rho } U_{m-1}\left( \frac{z+1}{2\rho \sqrt{z}} \right)  \right),
\end{array}
\right.
\end{equation}  
where \( U_{-1}(y) \equiv 0 \) and we take the branch \( \sqrt z \) that is positive for positive reals (of course, in our case \( \overline\alpha=\alpha\)). Observe that the map
\[
y(z) = (z+1)/(2\rho\sqrt z)
\]
takes \( \D \) into \( \{\mathrm{Re}(z)>0\}\setminus[0,1/\rho] \), the right half-plane with the real segment \( [0,1/\rho] \) removed, and its boundary values on \( \Delta_\alpha \) cover the real interval \( [0,1] \) twice. Therefore, 
\[ 
\sqrt{y(z)^2-1}= r(z)/(2\rho\sqrt z), \quad z\in\D.
\]
In particular, it follows from \eqref{varphi-asymp} that \eqref{b-value} holds. Observe that
\begin{equation}
\label{est-e}
|\epsilon(z)| = \left| \frac{y-\sqrt{y^2-1}}{y+\sqrt{y^2-1}}\right| = \left|y+\sqrt{y^2-1}\right|^{-2} <1
\end{equation}
 for \( |z|<1 \). Hence, \( b_{n+1}(z) \) converges pointwise and therefore locally uniformly (\(|b_{n+1}(z)|<1\) for \( z\in\D \)) to
\[
\frac{ z - (1+2\alpha) + r(z)}{ 1 - (1+2\alpha)z + r(z)  } = \frac{ z - (1+2\alpha) + r(z)}{ 1 - (1+2\alpha)z + r(z)  }\frac{ z - (1+2\alpha) - r(z)}{ z - (1+2\alpha) - r(z)} = \frac{-2\alpha}{r(z)+1-z}. \qedhere
\]
\end{proof}

\begin{lemma}
\label{lem:2}
Let \( h(x) := -\alpha(x+1)/r(x) \). It holds that
\begin{equation}
\label{hn-rep}
h_{n+1}(x) = h(x) \left(1 - \epsilon^{n+1}(x) \frac{\frac{n+1}\alpha\frac{(1-x)^2}x r(x) + 2R(x)(1-\epsilon^{n+1}(x))}{(1-\epsilon^{n+1}(x))(S(x)+R(x)\epsilon^{n+1}(x))}\right),
\end{equation}
where \( R(x) := r(x) + \alpha(1+x) \) and \( S(x) := r(x) - \alpha(1+x) \).
\end{lemma}
\begin{proof}
It follows from \eqref{b-value} that
\[
b_{n+1}(x) = 1  - \lambda \frac{(1-x)(1-\epsilon^{n+1}(x))}{D(x)},
\]
where \( \lambda := 2(1+\alpha) \) and  \( D(x) := \phi(x) - \lambda x - \epsilon^{n+1}(x)(\psi(x)-\lambda x) \). It can be readily checked that
\[
1- b_{n+1}^2(x) = 2\lambda\frac{(1-x)(1-\epsilon^{n+1}(x))(S(x)+R(x)\epsilon^{n+1}(x))}{D^2(x)}.
\]
Observe that
\[
D^\prime(x) = \phi^\prime(x) - \lambda - (n+1)\epsilon^n(x)\epsilon^\prime(x)(\psi(x)-\lambda x) - \epsilon^{n+1}(x)(\psi^\prime(x)-\lambda).
\]
It further holds that
\begin{eqnarray*}
b_{n+1}^\prime(x) & = & \lambda \frac{D(x)(1-\epsilon^{n+1}(x)+(n+1)(1-x)\epsilon^n(x)\epsilon^\prime(x)) + D^\prime(x)(1-x)(1-\epsilon^{n+1}(x))}{D^2(x)} \\
& =: & \lambda\frac{N_1(x) + (n+1)(1-x)\epsilon^n(x)\epsilon^\prime(x)N_2(x) + N_3(x)\epsilon^{n+1}(x) + N_4(x)\epsilon^{2(n+1)}(x) }{D^2(x)},
\end{eqnarray*}
where \( N_3(x), N_4(x) \) do not contain terms with \( \epsilon^\prime(x) \). We have that
\begin{eqnarray*}
N_1(x) &=& \phi(x)-\lambda x + (1-x)(\phi^\prime(x)-\lambda)  = -2\alpha + r(x) + r^\prime(x)(1-x) \\
&=& -2\alpha + 2\alpha^2(1+x)/r(x) = -2\alpha S(x)/r(x).
\end{eqnarray*}
Furthermore, we have that
\[
N_2(x) = D(x) - (\psi(x)-\lambda x)(1-\epsilon^{n+1}(x)) = 2r(x) = R(x) + S(x).
\]
It also holds that
\[
N_3(x) = - (\phi(x)-\lambda x) - (\psi(x) -\lambda x) - (1-x)(\psi^\prime(x)-\lambda+\phi^\prime(x)-\lambda) = 4\alpha.
\]
Finally, similarly to \( N_1(x) \), we have that
\[
N_4(x) = \psi(x) - \lambda x + (1-x)(\psi^\prime(x)-\lambda) = -2\alpha\big(R(x)/r(x)\big).
\]
Since
\begin{equation}
\label{e-prime} 
\epsilon^\prime(x) = \big((1-x)/x\big)\big(\epsilon(x)/r(x)\big),
\end{equation}
it follows from \eqref{set-up} that
\[
h_{n+1}(x) = h(x)\frac{(1-\epsilon^{n+1}(x))(S(x)-R(x)\epsilon^{n+1}(x)) -\frac{n+1}\alpha\frac{(1-x)^2}xr(x)\epsilon^{n+1}(x)}{(1-\epsilon^{n+1}(x))(S(x)+R(x)\epsilon^{n+1}(x))}
\]
from which the desired claim easily follows.
\end{proof}

\begin{lemma}
\label{lem:3}
Formula \eqref{hn-asymp} takes place.
\end{lemma}
\begin{proof}
It can be readily checked that the function \( |y+\sqrt{y^2-1}| \) is an increasing function of \( t \) for \( y=t \), \( t\in[1,\infty) \) and \( y=\pm\ic t \), \( t\in[0,\infty) \). Since \( \epsilon(1)=(1-|\alpha|)/(1+|\alpha|) \), it therefore holds that
\begin{eqnarray}
\max_{x\in[-1+(n+1)^{-1/2},1]}|\epsilon(x)|^n &=& \big|\epsilon(-1+(n+1)^{-1/2})\big|^n  \nonumber \\
\label{est-e-real}
&=& \left(1 - (n+1)^{-1/2}/\rho + \mathcal O\left((n+1)^{-1}\right)\right)^n \leq C_1e^{-\sqrt{n+1}/\rho}
\end{eqnarray}
for some absolute constant \( C_1>0 \). 

Assume that \( \alpha< 0 \). Then \( |S(x)|\geq r(x) \geq 2|\alpha|\rho \) for \( x\in[-1,1] \). Also, since \( |h(x)| \) is an increasing function on \( [-1,1] \), we have that \( |h(x)|\leq 1 \) for \( x\in[-1,1] \). Thus, we get from \eqref{hn-rep} and \eqref{est-e-real} that
\begin{eqnarray}
|h_{n+1}(x)-h(x)| &\leq& C_2(n+1)e^{-\sqrt{n+1}/\rho}\left((1-x)^2 + |R(x)|\right) \nonumber \\
\label{hnh1}
&\leq& C_3(1-x)^2 (n+1)e^{-\sqrt{n+1}/\rho}
\end{eqnarray}
for some absolute constants \( C_2,C_3 \), where one needs to observe that \( \epsilon(0)=0 \) and
\begin{equation}
\label{SR}
S(x)R(x) = \rho^2(1-x)^2.
\end{equation}
This proves the lemma in the case \( \alpha<0 \).

Suppose that \( \alpha >0 \). It is quite easy to see that estimate \eqref{hnh1} remains valid on \( [-1+(n+1)^{-1/2},0] \). Observe  also that \( \epsilon(x)>0 \) and is increasing for \( x\in(0,1] \), see \eqref{e-prime}, and \( 0< R(x) < 4\) on \( [-1,1] \). Then by using \eqref{SR} again, we get that
\begin{eqnarray*}
(1-\epsilon^{n+1}(x))(S(x)+R(x)\epsilon^{n+1}(x)) &\geq& S(x) - R(x)\epsilon^{2(n+1)}(x) \\
&\geq& (\rho^2/4)(1-x)^2 - 4\epsilon^{2(n+1)}(1)
\end{eqnarray*}
for \( x\in[0,1] \). Notice \( \delta_\alpha = \epsilon^{1/3}(1) \). Then
\[
(\rho^2/4)(1-x)^2 - 4\epsilon^{2(n+1)}(1) > (\rho^2/8)\delta_\alpha^{2(n+1)}
\]
 for \( x\in \big[0,1-\delta_\alpha^{(n+1)}\big] \) and \( n \) sufficiently large. Therefore, similarly to \eqref{hnh1}, it again follows from \eqref{SR} that there exists a constant \( C_4 \) such that
\[
|h_{n+1}(x)-h(x)| \leq C_4(1-x)^2(n+1) \big( \epsilon(1) / \delta_\alpha^2 \big)^{n+1} = C_4(1-x)^2(n+1)\epsilon^{2(n+1)/3}(1)
\]
 for \( x\in \big[0,1-\delta_\alpha^{(n+1)}\big] \). Since \( \epsilon(1)<1 \), the desired estimates follows.
\end{proof}

\section{Proof of Theorem~\ref{thm:1}}

To prove Theorem~\ref{thm:1} we shall use the following straightforward facts. If \( F(y) \) is analytic around the origin, then
\begin{equation}
\label{taylor}
F\left(\frac t{n+1}\right) = \sum_{p=0}^{N-1}\frac{F_pt^p}{(n+1)^p} + \frac{\widetilde F_N(t)t^N}{(n+1)^N}, \quad \big| \widetilde F_N(t)\big| \leq C_F^{N+1}, \quad 
\end{equation}
for \( t\in I_n:=\big[0,\sqrt{n+1}\big]\) and all \( n\geq n_F \), where \( F_p = F^{(p)}(0)/p! \), the last estimate follows from the extended Cauchy integral formula, and \( C_F \) is independent of \( n,N \). Further, if functions \( u(t),v(t) \) satisfy
\begin{equation} 
\label{product0}
g(t) = \sum_{p=0}^{N-1} \frac{B_p(g;t)}{(n+1)^p} + \frac{\widetilde B_N(g;t)}{(n+1)^N},
\end{equation}
with \( g\in\{u,v\} \), then so does their product and
\begin{equation} 
\label{product1}
B_p(uv;t) = \sum_{k=0}^p B_k(u;t)B_{p-k}(v;t)
\end{equation}
for \( p\leq N-1\), while
\begin{equation} 
\label{product2}
\widetilde B_N(uv;t) = \sum_{l=0}^N \frac1{(n+1)^l}\sum_{k+m=N+l,~k,m\leq N} B_{N,k}(u;t) B_{N,m}(v;t)
\end{equation}
with \( B_{N,k}(g;t) = B_k(g;t) \) for \( k<N \) and \( B_{N,N}(t) = \widetilde B_N(g;t) \). Finally, let \( F(y) \) be as in \eqref{taylor} and \( g(t) \) be as in \eqref{product0} with \( B_0(g;t)=0 \). Assume that the values of \( g(t) \) lie the domain of holomorphy of \( F(y) \) for all \( n\geq n_g \). Then
\begin{equation}
\label{taylor1}
F(g(t)) = F(0) + \sum_{p=1}^{N-1} \frac{B_p(F\circ g;t)}{(n+1)^p} + \frac{\widetilde B_N(F\circ g;t)}{(n+1)^N},
\end{equation}
with
\begin{equation}
\label{taylor2}
B_p(F\circ g;t) = \sum \frac{F^{(m)}(0)}{m_1!\cdots m_{N-1}!}\prod_{k=1}^{N-1}B_k^{m_k}(g;t)
\end{equation}
where \( m=m_1+\cdots + m_{N-1} \) and the sum is taken over all partitions \( p = \sum_{i=1}^{N-1} i m_i \), \( m_i\geq0 \), and
\begin{equation}
\label{taylor3}
\widetilde B_N(F\circ g;t) = \sum_{l=0}^{N(N-1)} \frac1{(n+1)^l}\sum \frac{F^{(m)}(0)}{m_1!\cdots m_N!}\prod_{k=1}^N B_{N,k}^{m_k}(g;t)
\end{equation}
where \( m=m_1+\cdots + m_N \), the inner sum is taken over all partitions \( l+N = \sum_{i=1}^N i m_i \), \( m_i\geq0 \), and \( B_{N,k}(g;t) \) has the same meaning as in \eqref{product2}.

\begin{lemma}
\label{lem:4}
Let \( t\in I_n =\big[0,\sqrt{n+1}\big] \). Then it holds for all \( N\geq1 \) that
\begin{equation}
\label{e-11}
r\left(-1+\frac t{n+1}\right) = 2\rho\left(\sum_{p=0}^{N-1}\frac{r_p t^p}{(n+1)^p} + \frac{\tilde r_N(t) t^N}{(n+1)^N}\right)
\end{equation}
for some constants \( r_p \) and functions \( \tilde r_N(t) \) that obey estimate in \eqref{taylor}. In particular, \( r_0=1 \) , \( r_1=-1/2 \), \( r_2 = (1-\rho^2)/(8\rho^2) \). Moreover, for \( \epsilon(z) \), defined in Lemma~\ref{lem:1}, it holds that
\begin{equation}
\label{e-en}
\epsilon^{n+1}\left(-1+\frac t{n+1}\right) = (-1)^{n+1}e^{-t/\rho}\left( 1 + \sum_{p=1}^{N-1} \frac{t^{p+1}e_p(t)}{(n+1)^p} + \frac{t^{N+1}\tilde e_N(t)}{(n+1)^N} \right),
\end{equation}
where \( e_p(t) \) is a polynomial of degree \( p-1 \) independent of \( n,N \), in particular, \( e_1(t)\equiv -1/(2\rho) \), and \( |\tilde e_N(t)| \) is bounded above on \( I_n \) by a polynomial of degree \( N-1 \) whose coefficients depend only on \( N \).
\end{lemma}
\begin{proof}
Observe that for \( y>0 \) it follows from \eqref{r} and the choice of the branch of \( r(z) \) that
\[
r(-1+y) = 2\rho\sqrt{1-y+y^2/(4\rho^2)},
\]
where the root in right-hand side of the above equality is principal. Since the right-hand side above is analytic around the origin, expansion \eqref{e-11} follows from \eqref{taylor}. An absolutely analogous argument yields the expansion
\[
\log \left(-\epsilon\left(-1+\frac t{n+1}\right)\right) = \sum_{p=1}^N\frac{\epsilon_p t^p}{(n+1)^p} + \frac{\tilde \epsilon_{N+1}(t)t^{N+1}}{(n+1)^{N+1}}, \quad \epsilon_1 = -\frac1\rho,~~\epsilon_2 = -\frac1{2\rho},
\]
where \( |\tilde \epsilon_{N+1}(t)| \) has an upper bound as in \eqref{taylor}. Since we can write
\[
\epsilon^{n+1}\left(-1+\frac t{n+1}\right) = \frac{(-1)^{n+1}}{e^{t/\rho}} \exp\left\{(n+1)\left(\log \left(-\epsilon\left(-1+\frac t{n+1}\right)\right) + \frac1\rho \frac t{n+1}\right)\right\},
\]
it follows from \eqref{taylor1}--\eqref{taylor3} that \eqref{e-en} holds, where \( e_p(t) \) is a polynomial of degree \( p-1 \) independent of \( n,N \) (notice that always \( m\leq p \) in \eqref{taylor2}) and \( |\tilde e_N(t)| \) is bounded above on \( I_n \) by a polynomial of degree \( N-1 \) whose coefficients depend only on \( N \) (again, we use that \( m\leq l+N \) in \eqref{taylor3} and that \( t^{2l}\leq (n+1)^l \) on \( I_n \)).
\end{proof}

\begin{lemma}
\label{lem:5}
Set \( \gamma(s) := 2s/(e^s-e^{-s}) \) and let \( x=-1+t/(n+1) \), \( t\in I_n \). It holds that
\begin{equation}
\label{hn-Gam}
h_{n+1}(x) = h(x) - (-1)^{n+1}\frac{(1-x)^2}4\gamma(t/\rho)(1+ \Gamma_{n+1}(t))
\end{equation}
with \( \Gamma_{n+1}(t) \) having an expansion of the form
\begin{equation}
\label{Gam}
\Gamma_{n+1}(t) = \sum_{p=1}^{N-1}\frac{H_p(t)}{(n+1)^p} + \frac{\widetilde H_N(t)}{(n+1)^N},
\end{equation}
for any \( N\geq2 \), where \( H_1(t)=t-(-1)^{n+1}(\alpha/2\rho)t + \mathcal O(t^2) \), \( H_p(t)= \mathcal O(t^2) \), \( p\geq2\), and \( \widetilde H_N(t)=\mathcal O(t^2)\) as \( t\to0 \), \( |H_p(t)| \) is bounded above by a polynomial of degree \( 2p \) independent of \( n,N \), while \( |\widetilde H_N(t)| \) is bounded above on \( I_n \) by a polynomial of degree \( 2N \) whose coefficients depend on \( N \) but not on \( n \). 
\end{lemma}
\begin{proof}
Recall \eqref{hn-rep}.  Notice that
\begin{equation}
\label{5:0}
\big(1-\epsilon^{n+1}(x)\big)\big(S(x) + R(x)\epsilon^{n+1}(x)\big) = S(x) + 2\alpha(x+1)\epsilon^{n+1}(x) - R(x)\epsilon^{2(n+1)}(x).
\end{equation}
It follows from \eqref{e-11} that  \( S(x) \) and \( R(x) \) have expansions as in \eqref{product0} with
\[
B_p(S;t) = B_p(R;t) = 2\rho r_pt^p, ~~~p\neq1, \quad B_1(S;t) = -(\alpha+\rho)t,~~~B_1(R;t) = (\alpha-\rho)t,
\]
and \( \widetilde B_N(S;t) = \widetilde B_N(R;t) = 2\rho\tilde r_N(t) \) for any \( N\geq2 \). Therefore, we get from \eqref{product1}--\eqref{product2} and \eqref{e-en} that
\[
R(x)\epsilon^{2(n+1)}(x) = 2\rho e^{-2t/\rho}\left(1 + \sum_{p=1}^{N-1}\frac{C_p(t)t^p}{(n+1)^p} + \frac{\widetilde C_N(t)t^N}{(n+1)^N}\right)
\]
for any \( N\geq2 \), where \( C_1(t) =(\alpha-\rho-2t)/(2\rho) \), \( C_p(t) = r_p + tq_p(t) \) for some polynomial \( q_p(t) \) of degree \( p-1 \) when \( p\geq2 \), and \( |\widetilde C_N(t)| \) is bounded above on \( I_n \) by a polynomial of degree \( N \) independent of \( n \). Consequently, we get that the expression in \eqref{5:0} has an expansion
\[
2\rho\big(1 - e^{-2t/\rho}\big)\left(1 + \frac{D_1(t)}{n+1} + \sum_{p=2}^{N-1}\frac{D_p(t)t^p}{(n+1)^p} + \frac{\widetilde D_N(t)t^N}{(n+1)^N} \right)
\]
for all \( N\geq2 \), where
\begin{eqnarray*}
D_1(t) &=& -\alpha\left(\frac{1-(-1)^{n+1}e^{-t/\rho}}2\right)^2\frac{2t/\rho}{1-e^{-2t/\rho}} + \frac t2\left(\frac{2t/\rho}{e^{2t/\rho}-1} -1 \right) \\
&=& -\alpha\frac{1-(-1)^{n+1}}2 + \mathcal O\left(t^2\right) \quad \text{as} \quad t\to0,
\end{eqnarray*}
with \( |D_1(t)| \) bounded above by a linear function independent of \( n,N \), and
\[
D_p(t) = r_p + \gamma(t/\rho)\left((-1)^{n+1}\alpha e_{p-1}(t) - \rho e^{-t/\rho}q_p(t)\right)/2
\]
for all \( p\geq2 \), with \( |D_p(t)| \) being bounded above on \( [0,\infty) \), and \( |\widetilde D_N(t)| \) that is bounded on \( I_n \) by a constant that depends on \( N \) but not on \( n \). In particular, we have that
\[
\left|\frac{D_1(t)}{n+1} + \sum_{p=2}^{N-1}\frac{D_p(t)t^p}{(n+1)^p} + \frac{\widetilde D_N(t)t^N}{(n+1)^N} \right| = \left|\frac{D_1(t)}{n+1} + \frac{\widetilde D_2(t)t^2}{(n+1)^2} \right| < \frac{c_N}{\sqrt{n+1}}< 1
\]
for \( t\in I_n \) and all \( n\geq n_N \), where \( c_N,n_N \) are constants dependent only on \( N \). Thus, it follows from \eqref{taylor1}--\eqref{taylor3} with \( F(y) = 1/(1+y) \) that the reciprocal of \eqref{5:0} has an expansion
\begin{equation}
\label{5:1}
\frac1{2\rho}\frac1{1-e^{-2t/\rho}}\left(1 + \sum_{p=1}^{N-1}\frac{E_p(t)}{(n+1)^p} + \frac{\widetilde E_N(t)}{(n+1)^N} \right),
\end{equation}
for all \( N\geq2 \), where \( E_1(t) = -D_1(t) \) and more generally
\begin{equation}
\label{5:3}
E_p(t) = (-1)^p D_1^p(t) + \mathcal O\left(t^2\right) = \alpha^p\frac{1-(-1)^{n+1}}2 + \mathcal O\left(t^2\right)
\end{equation}
as \( t\to0 \) with \( |E_p(t)| \) bounded above by a polynomial of degree \( p \) independent of \( n,N \), while \( |\widetilde E_N(t)| \) is bounded above on \( I_n \) by a polynomial of degree \( N \) whose coefficients depend on \( N \) but not on \( n \). Furthermore, observe that
\begin{multline*}
-h(x)\epsilon^{n+1}(x)\left(\frac{n+1}\alpha\frac{(1-x)^2}x r(x) + 2R(x)\big(1-\epsilon^{n+1}(x)\big)\right) = \\ -(n+1)(1+x)(1-x)^2\epsilon^{n+1}(x)\left(-\frac1x - \frac{2\alpha}{n+1} \frac{R(x)}{r(x)}\frac{1-\epsilon^{n+1}(x)}{(1-x)^2}\right).
\end{multline*}
It follows from an argument similar to the one given in the first part of the lemma that the above expression has an expansion of the form
\begin{equation}
\label{5:2}
-(1-x)^2(-1)^{n+1}te^{-t/\rho}\left(1+ \sum_{p=1}^{N-1}\frac{G_p(t)t^{p-1}}{(n+1)^p} + \frac{\widetilde G_N(t)t^{N-1}}{(n+1)^N}\right),
\end{equation}
for any \( N\geq 3 \), where
\begin{equation}
\label{5:4}
G_1(t) = -\alpha \frac{1-(-1)^{n+1}}2 + \left(1- (-1)^{n+1}\frac{\alpha}{2\rho}\right)t + \mathcal O\left(t^2\right)
\end{equation}
and
\begin{equation}
\label{5:5}
G_2(t) = -\alpha\frac{1-(-1)^{n+1}}2\left(1+\frac{\alpha}{2\rho}\right) + \mathcal O(t)
\end{equation}
as \( t\to0 \), \( |G_p(t)| \) is bounded above by a polynomial of degree \( p+1 \) independent of \( n,N \), while \( |\widetilde G_N(t)| \) is bounded above on \( I_n \) by a polynomial of degree \( N+1 \) whose coefficients depend on \( N \) but not on \( n \). We now get from \eqref{hn-rep}, \eqref{5:1}, and \eqref{5:2}, that \eqref{hn-Gam} and \eqref{Gam} do hold for \( N\geq3 \) and functions \( H_p(t) \) and \( \widetilde H_N(t) \) that can be computed via \eqref{product1}--\eqref{product2} and whose moduli satisfy the described bounds. The vanishing of \( H_p(t) \) as \( t\to0 \) can be verified by using \eqref{product1}, \eqref{5:3}, \eqref{5:4}, and \eqref{5:5}. To see that \( \widetilde H_N(t)=\mathcal O(t^2) \), observe that
\[
h_{n+1}(x) = -(-1)^{n+1} - (-1)^{n+1}(1-t/(n+1))\widetilde H_N(t)(n+1)^{-N} + \mathcal O\left(t^2\right)
\] 
by what precedes. Thus, we need to show that \( h_{n+1}(x) + (-1)^{n+1} \) is divisible by \( (1+x)^2 \) (of course, if this were not true, formula \eqref{Enalpha} would not have made sense). Since \( h_{n+1}(-1)=-(-1)^{n+1} \), it must hold that \( h_{n+1}^\prime(-1) =0 \). As was mentioned before \eqref{Enalpha}, \( h_{n+1}(x) = h_{n+1}(1/x) \) and therefore \( x^2 h_{n+1}^\prime(x) = - h_{n+1}^\prime(1/x) \), which yields the desired claim. Finally, since \( \widetilde H_2(t) = H_2(t) + \widetilde H_3(t)(n+1)^{-1} \), we can take \( N=2 \) in \eqref{Gam} as well.
\end{proof}

\begin{lemma}
\label{lem:6}
let \( x=-1+t/(n+1) \), \( t\in I_n \). It holds that
\begin{equation}
\label{hn-K}
\frac{\sqrt{1-h_{n+1}^2(x)}}{1-x} = \frac{\rho f(t/\rho)}{r(x)}\left(1+\gamma(t/\rho)\sum_{p=1}^{N-1}\frac{K_p(t)}{(n+1)^p} + \gamma(t/\rho) \frac{\widetilde K_N(t)}{(n+1)^N}\right)
\end{equation}
for any \( N\geq 2 \), where \( |K_p(t)| \) is bounded above by a polynomial of degree \( 2p \) independent of \( n,N \) while \( |\widetilde K_N(t)| \) is bounded above on \( I_n \) by a polynomial of degree \( 2N \) whose coefficients depend on \( N \) but not on \( n \). 
\end{lemma}
\begin{proof}
Observe that \( 1-h^2(x) = \rho^2(1-x)^2r^{-2}(x) \). Then it follows from \eqref{hn-Gam} that
\begin{multline*}
\frac{1-h_{n+1}^2(x)}{1-h^2(x)} = 1 +(-1)^{n+1}\gamma(t/\rho)\big(1+\Gamma_{n+1}(t)\big)h(x)\frac{r^2(x)}{2\rho^2} - \\ \gamma(t/\rho)^2\big(1+\Gamma_{n+1}(t)\big)^2\frac{(1-x)^2}{4}\frac{r^2(x)}{4\rho^2}.
\end{multline*}
Since \( h(x)r(x)=-\alpha(1+x) \), expansions \eqref{e-11}, \eqref{Gam} and formulae \eqref{product1}--\eqref{product2} yield that
\[
(-1)^{n+1}\big(1+\Gamma_{n+1}(t)\big)h(x)\frac{r^2(x)}{2\rho^2} = \sum_{p=1}^{N-1}\frac{H_p^*(t)}{(n+1)^p} + \frac{\widetilde H_N^*(t)}{(n+1)^N},
\]
for any \( N\geq2 \), where \( H_1^*(t)=-(-1)^{n+1}(\alpha/\rho)t \), \( H_p^*(t)= \mathcal O(t^2) \), \( p\geq2\), and \( \widetilde H_N^*(t)=\mathcal O(t^2)\) as \( t\to0 \), while \( |H_p^*(t)| \) and \( |\widetilde H_N^*(t)| \) have similar bounds to \( |H_p(t)| \) and \( |\widetilde H_N(t)| \). Furthermore, it clearly holds that
\[
\frac{(1-x)^2}4 = 1- \frac t{n+1} + \frac14\frac{t^2}{(n+1)^2} \quad \text{and} \quad \frac{r^2(x)}{4\rho^2} = 1 - \frac t{n+1} + \frac1{4\rho^2}\frac{t^2}{(n+1)^2}.
\]
Therefore, we again get from \eqref{product1}--\eqref{product2} that
\[
\big(1+\Gamma_{n+1}(t)\big)^2\frac{(1-x)^2}{4}\frac{r^2(x)}{4\rho^2} = 1 + \sum_{p=1}^{N-1}\frac{H_p^{**}(t)}{(n+1)^p} + \frac{\widetilde H_N^{**}(t)}{(n+1)^N},
\]
for any \( N\geq2 \), where \( H_1^{**}(t)=-(-1)^{n+1}(\alpha/\rho)t + \mathcal O(t^2) \), \( H_p^{**}(t)= \mathcal O(t^2) \), \( p\geq2\), and \( \widetilde H_N^{**}(t)=\mathcal O(t^2)\) as \( t\to0 \) while \( |H_p^{**}(t)| \) and \( |\widetilde H_N^{**}(t)| \) have similar bounds to \( |H_p(t)| \) and \( |\widetilde H_N(t)| \). Altogether, it holds that
\[
\frac{1-h_{n+1}^2(x)}{1-h^2(x)} = f^2(t/\rho)\left(1+\gamma(t/\rho)\sum_{p=1}^{N-1}\frac{J_p(t)}{(n+1)^p} + \gamma(t/\rho)\frac{\widetilde J_N(t)}{(n+1)^N}\right),
\]
where \( J_p(t) = f^{-2}(t/\rho) \big(H_p^*(t) - \gamma(t/\rho)H_p^{**}(t)\big) \) and a similar formula holds for \( \widetilde J_N(t) \). Observe that \( f^2(s) \) is a positive function for \( s>0 \) that tends to \( 1 \) as \( s\to\infty \) and such that \( f^2(s) = s^2/3 + \mathcal O(s^4) \) as \( s\to0 \). Therefore, it follows from the corresponding properties of \( H_p^*(t) \), \( H_p^{**}(t) \), \( \widetilde H_N^*(t) \), and \( \widetilde H_N^{**}(t) \) that \( J_p(t) \) and \( \widetilde J_N(t) \) have finite value at the origin and have moduli that satisfy similar bounds to \( |H_p(t)| \) and \( |\widetilde H_N(t)| \). Observe also that there exist \( n_N \) and \( c_N<1 \) such that
\[
\left|\gamma(t/\rho)\sum_{p=1}^{N-1}\frac{J_p(t)}{(n+1)^p} + \gamma(t/\rho)\frac{\widetilde J_N(t)}{(n+1)^N}\right| < c_N
\]
for all \( n\geq n_N \). Therefore, the claim of the lemma now follows from \eqref{taylor1}--\eqref{taylor3} applied with \( F(y)=\sqrt{1+y} \).
\end{proof}

\begin{lemma}
\label{lem:7}
Let \( x=-1+t/(n+1) \), \( t\in I_n \). There exist constants \( O_p \), \( p\geq 1 \), such that
\begin{multline*}
\frac{2\rho }\pi\int_0^{\sqrt{n+1}} \frac{f(t/\rho)}{t r(x)}  \dd t = \frac1{2\pi}\log(n+1)+\frac{A_0}2 - \frac1\pi\log(2\rho) + \frac1\pi \mathcal L\left(-1+\frac1{\sqrt{n+1}}\right) \\ +\sum_{p=1}^{N-1} \frac{O_p}{(n+1)^p} + \mathcal O_N\left((n+1)^{-N}\right), \quad 
\end{multline*}
for any \( N\geq1 \), where \( \mathcal O_N(\cdot) \) does not depend on \( n \) and
\[
\mathcal L(x) := \log\left(\frac{4\rho}{\rho(1-x)+r(x)}\right).
\]
\end{lemma}
\begin{proof}
Similarly to \eqref{e-11}, there exist constants \( r_p^* \) such that
\begin{equation}
\label{7:1}
\frac{2\rho}{r(x)} = 1 + \sum_{p=1}^{N-1} \frac{r_p^*t^p}{(n+1)^p} + \frac{\tilde r^*_N(t)t^N}{(n+1)^N},
\end{equation}
for any \( N\geq 1 \), where \( |\tilde r^*_N(t)| \) is bounded above on \( I_n \) by a constant that depends only on \( N \). Then
\[
\mathcal I_1 := \frac{2\rho }\pi\int_0^\rho \frac{f(t/\rho)}{t r(x)}  \dd t = \frac1\pi \int_0^1\frac{f(t)}{t}\dd t + \sum_{p=1}^{N-1} \frac{L_p}{(n+1)^p} + \mathcal O_N\left((n+1)^{-N}\right),
\]
where \( L_p := (r_p^*\rho^p/\pi)\int_0^1 f(t)t^{p-1}\dd t\) and \( \mathcal O_N(\cdot) \) does not depend on \( n \). Furthermore, it holds that
\begin{equation}
\label{7:2}
\mathcal I_2 :=  \frac{2\rho }\pi\int_\rho^{\sqrt{n+1}} \frac{\dd t}{t r(x)} = \frac{2\rho}\pi \int_{-1+\rho/(n+1)}^{-1+1/\sqrt{n+1}} \frac{dx}{(1+x)r(x)}.
\end{equation}
It can be easily verified by differentiation that an antiderivative of \( 2\rho/((1+x)r(x)) \) is \( \log(1+x) + \mathcal L(x) \). Again, similarly to \eqref{e-11}, there exist constants \( l_p \) such that
\[
\mathcal L(x) = \sum_{p=1}^{N-1} \frac{l_pt^p}{(n+1)^p} + \frac{\tilde l_N(t)t^N}{(n+1)^N},
\]
for any \( N\geq 1 \), where \( |\tilde l_N(t)| \) is bounded above on \( I_n \) by a constant that depends only on \( N \). Therefore, it holds that
\[
\mathcal I_2 = \frac1{2\pi}\log(n+1) - \frac1\pi\log\rho + \frac1\pi\mathcal L\left(-1+\frac1{\sqrt{n+1}}\right) - \sum_{p=1}^{N-1} \frac{l_p\rho^p/\pi}{(n+1)^p} + \mathcal O_N\left((n+1)^{-N}\right),
\]
where, again, \( \mathcal O_N(\cdot) \) does not depend on \( n \). Next, we have from \eqref{7:1} that
\begin{multline*}
\mathcal I_3 := \frac{2\rho }\pi\int_\rho^{\sqrt{n+1}} \frac{f(t/\rho)-1}{t r(x)}\dd t = \frac1\pi\int_1^{\sqrt{n+1}/\rho}\frac{f(t)-1}t \dd t + \\ \sum_{p=1}^{N-1} \frac{r_p^*\rho^p/\pi}{(n+1)^p}\int_1^{\sqrt{n+1}/\rho}(f(t)-1)t^{p-1} \dd t + \frac{\rho^N/\pi}{(n+1)^N} \int_1^{\sqrt{n+1}/\rho}(f(t)-1)\tilde r_N^*(\rho t)t^{N-1} \dd t
\end{multline*}
for any \( N\geq1 \). Notice that
\begin{equation}
\label{7:3}
0< 1 - f(t) < t^2\mathrm{csch}^2(t) < 8t^2e^{-2t}, \quad t\geq 1.
\end{equation}
Therefore, it holds that
\begin{equation}
\label{7:4}
0 < \int_{\sqrt{n+1}/\rho}^\infty (1-f(t))t^{p-1}\dd t  \leq  C_p(n+1)^{(p+1)/2}e^{-(2/\rho)\sqrt{n+1}} = o_N\left( (n+1)^{-N} \right)
\end{equation}
for any \( p\geq0 \) and \( N\geq 1 \) and some constant \( C_p \) that depends only on \( p \), where \( o_N(\cdot) \) does not depend on \( n \). Moreover, since \( |\tilde r^*_N(t)| \) is bounded above on \( I_n \) by a constant that depends only on \( N \), we have that
\[
\left|\int_1^{\sqrt{n+1}/\rho}(f(t)-1)\tilde r_N^*(\rho t)t^{N-1} \dd t\right| \leq C_N^*\int_1^\infty(1-f(t))t^{N-1} \dd t = C_N^{**}.
\]
Thus, we can conclude that
\[
\mathcal I_3 = \frac1\pi\int_1^\infty\frac{f(t)-1}t \dd t + \sum_{p=1}^{N-1}\frac{M_p}{(n+1)^p} + \mathcal O_N\left( (n+1)^{-N} \right),
\]
where \( M_p := (r_p^*\rho^p/\pi)\int_1^\infty (f(t)-1)t^{p-1}\dd t\) and \( \mathcal O_N(\cdot) \) does not depend on \( n \). Since the integral in the statement of the lemma is equal to \( \mathcal I_1 + \mathcal I_2 + \mathcal I_3 \), the desired claim now follows from the definition of \( A_0 \) in \eqref{A0}, where \( O_p = L_p - l_p\rho^p/\pi + M_p \).
\end{proof}

\begin{lemma}
\label{lem:8}
There exist constants \( T_p \) such that
\begin{multline*}
\frac2\pi\int_{-1}^{-1+1/\sqrt{n+1}}\frac{\sqrt{1-h_{n+1}^2(x)}}{1-x^2}\dd x = \frac1{2\pi}\log(n+1)+\frac{A_0}2 - \frac1\pi\log(2\rho) +  \\ \frac1\pi \mathcal L\left(-1+\frac1{\sqrt{n+1}}\right) + \sum_{p=1}^{N-1} \frac{T_p}{(n+1)^p} + \mathcal O_N\left((n+1)^{-N}\right), \quad 
\end{multline*}
for any \( N\geq1 \), where \( \mathcal O_N(\cdot) \) does not depend on \( n \).
\end{lemma}
\begin{proof}
Recall \eqref{hn-K}. It follows from \eqref{7:1} and \eqref{product1}--\eqref{product2} that
\[
\frac{2\rho}{r(x)}\left(\sum_{p=1}^{N-1}\frac{K_p(t)}{(n+1)^p} + \frac{\widetilde K_N(t)}{(n+1)^N}\right) = \sum_{p=1}^{N-1}\frac{S_p(t)}{(n+1)^p} + \frac{\widetilde S_N(t)}{(n+1)^N}
\]
for any \( N\geq 2 \), where \( |S_p(t)| \) is bounded above by a polynomial of degree \( 2p \) independent of \( n,N \) while \( |\widetilde S_N(t)| \) is bounded above on \( I_n \) by a polynomial of degree \( 2N \) whose coefficients depend on \( N \) but not on \( n \). Similarly to \eqref{7:3}, it holds that \( \gamma(s) < 3se^{-s} \) for \( s\geq\log2 \). Because \( f(s)\to1 \) as \( s\to\infty \), it holds as in \eqref{7:4} that
\[
0 < \int_{\sqrt{n+1}/\rho}^\infty |\rho S_p(\rho t)| \gamma(t)f(t)\dd t  \leq  C_p(n+1)^{p+1/2} e^{-\sqrt{n+1}/\rho} = o_N\left( (n+1)^{-N} \right)
\]
for any \( p\geq1 \) and \( N\geq 1 \) and some constant \( C_p \) that depends only on \( p \), where \( o_N(\cdot) \) does not depend on \( n \). Moreover, a similar estimate takes place if \( S_p(t) \) is replaced by \( \widetilde S_N(t) \). The claim of the lemma now follows by making a substitution \( x=-1+t/(n+1) \) to get
\[
\frac2\pi\int_{-1}^{-1+1/\sqrt{n+1}}\frac{\sqrt{1-h_{n+1}^2(x)}}{1-x^2}\dd x = \frac2\pi\int_0^{\sqrt{n+1}}\frac{\sqrt{1-h_{n+1}^2(x)}}{1-x}\frac{\dd t}t
\]
and then using Lemmas~\ref{lem:6} and~\ref{lem:7}, where \( T_p = O_p + (\rho/\pi) \int_0^\infty f(t)\gamma(t)S_p(\rho t)\dd t \) (since \( T_1/(n+1) = \mathcal O_N((n+1)^{-1})\), the claim indeed holds for all \( N\geq1 \)).
\end{proof}

\begin{lemma}
\label{lem:9}
It holds that
\begin{multline*}
 \frac2\pi\int_{-1+1/\sqrt{n+1}}^{1-\delta^{n+1}_\alpha} \frac{\sqrt{1-h_{n+1}^2(x)}}{1-x^2}\, \dd x = \frac1{2\pi} \log(n+1) + \\  \frac1\pi\log\left(\frac{4\rho}{|\alpha|}\right) - \frac1\pi\mathcal L\left(-1+\frac1{\sqrt{n+1}}\right) + o_N\left( (n+1)^{-N} \right)
\end{multline*}
for any integer \( N\geq1 \), where \( o_N(\cdot) \) is independent of \( n \), but does depend on \( N \).
\end{lemma}
\begin{proof}
Since \( |h_{n+1}(x)|,|h(x)|\leq 1 \) when \( x\in[-1,1] \), it holds that
\begin{eqnarray*}
\left|\frac{ \sqrt{1-h^2(x)} - \sqrt{1-h_{n+1}^2(x)}}{1-x^2}\right| & = & \frac{\big|h_{n+1}^2(x) - h^2(x)\big|}{ (1-x^2)\big( \sqrt{1-h^2(x)} + \sqrt{1-h_{n+1}^2(x)} \big) } \\
&\leq& \frac{2|h_{n+1}(x) - h(x)|}{ (1-x^2)\sqrt{1-h^2(x)}} = \frac2\rho\frac{r(x)}{(1+x)} \frac{|h_{n+1}(x) - h(x)|}{(1-x)^2}.
\end{eqnarray*}
Since \( r(x) \leq 2 \), \( x\in[-1,1] \), we obtain from \eqref{hn-asymp} that
\[
\left|\frac{ \sqrt{1-h^2(x)} - \sqrt{1-h_{n+1}^2(x)}}{1-x^2}\right| \leq C(n+1)^{3/2}e^{-\sqrt{n+1}/\rho}
\]
for \( -1+1/\sqrt{n+1}\leq x\leq 1-\delta^{n+1}_\alpha \) and some constant \( C \). Therefore, it holds that
\[
\left|\frac2\pi\int_{-1+1/\sqrt{n+1}}^{1-\delta_\alpha^{n+1}}\frac{ \sqrt{1-h^2(x)} - \sqrt{1-h_{n+1}^2(x)}}{1-x^2}\, \dd x \right| = o_N\left( (n+1)^{-N} \right)
\]
for any \( N\geq1 \), where \( o_N(\cdot) \) is independent of \( n \), but does depend on \( N \). Furthermore, since \( r(x)\geq 2|\alpha|\rho \) for \( x\in[-1,1] \), it holds that
\[
 \frac 2\pi \int_{1-\delta^{n+1}_\alpha}^1\frac{\sqrt{1-h^2(x)}}{1-x^2} \dd x =  \frac 2\pi \int_{1-\delta^{n+1}_\alpha}^1 \frac{\rho\dd x}{(1+x)r(x)} \leq \frac{\delta^{n+1}_\alpha}{|\alpha|\pi} = o_N\left( (n+1)^{-N} \right)
\]
for any \( N\geq1 \) by the very definition of \( \delta_\alpha \), where, again, \( o_N(\cdot) \) is independent of \( n \), but does depend on \( N \). The observation made after \eqref{7:2} allows us now to conclude that
\[
\frac 2\pi \int_{-1+1/\sqrt{n+1}}^1 \frac{\rho\dd x}{(1+x)r(x)} = \frac1{2\pi} \log(n+1) + \frac1\pi\log\left(\frac{4\rho}{|\alpha|}\right) - \frac1\pi \mathcal L\left(-1+\frac1{\sqrt{n+1}}\right),
\]
which finishes the proof of the lemma.
\end{proof}

\begin{lemma}
\label{lem:10}
When \( \alpha>0 \), it holds that
\[
\frac2\pi\int_{1-\delta_\alpha^{n+1}}^1 \frac{\sqrt{1-h_{n+1}^2(x)}}{1-x^2}\, \dd x = 1 + o_N\left( (n+1)^{-N} \right)
\]
for any \( N\geq1 \), where \( o_N(\cdot) \) is independent of \( n \), but does depend on \( N \).
\end{lemma}
\begin{proof}
It follows from \eqref{hn-rep} and \eqref{SR} that
\[
h_{n+1}(x) = h(x) - h(x)\frac{\epsilon^{n+1}(x)X_{n+1}(x)}{(1-x)^2+\epsilon^{n+1}(x)Y_{n+1}(x)},
\]
where
\[
X_{n+1}(x) := \frac{R(x)}{\alpha\rho^2}\left((n+1)r(x)\frac{(1-x)^2}x + 2\alpha R(x)\big(1-\epsilon^{n+1}(x)\big)\right)
\]
and
\[
Y_{n+1}(x) :=\frac{R(x)}{\rho^2}\left(2\alpha(x+1) - R(x)\epsilon^{n+1}(x)\right).
\]
Therefore, we can write
\begin{multline*}
1-h_{n+1}^2(x) = \rho^2\frac{(1-x)^2}{r^2(x)} + h^2(x)\frac{(1-x)^2\epsilon^{n+1}(x)X_{n+1}(x)}{((1-x)^2+\epsilon^{n+1}(x)Y_{n+1}(x))^2} \times \\ \left(2 - \epsilon^{n+1}(x)\frac{X_{n+1}(x)-2Y_{n+1}(x)}{(1-x)^2}\right).
\end{multline*}
We have that
\[
\frac{X_{n+1}(x) - 2Y_{n+1}(x)}{(1-x)^2} = \frac{R(x)}{\rho^2(1-x)^2}\left(2S(x)  + (n+1)\frac{(1-x)^2r(x)}{\alpha x}\right) = 2+ (n+1)\frac{r(x)R(x)}{\alpha\rho^2 x},
\]
where we used \eqref{SR} once more. Hence, it holds that
\begin{equation}
\label{dens}
\frac{\sqrt{1-h_{n+1}^2(x)}}{1-x^2} = \frac{\epsilon^{(n+1)/2}(x)V_{n+1}(x)}{(1-x)^2+\epsilon^{n+1}(x)Y_{n+1}(x)},
\end{equation}
where
\begin{multline*}
V_{n+1}(x) := -\frac2\rho\frac{h(x)R(x)}{1+x}\sqrt{\left(1-\epsilon^{n+1}(x)\left(1+(n+1)\frac{r(x)R(x)}{2\alpha\rho^2 x}\right)\right)\times} \\
\overline{\times\left(1-\epsilon^{n+1}(x)+(n+1)\frac{(1-x)^2r(x)}{2\alpha xR(x)}\right) + \left(\rho^2\frac{(1-x)^2+\epsilon^{n+1}(x)Y_{n+1}(x)}{2\epsilon^{(n+1)/2}(x)h(x)r(x)R(x)}\right)^2}
\end{multline*}
(observe that \( -h(x)>0 \)). Recall that \( \delta_\alpha = \epsilon^{1/3}(1) \). In particular, we get from \eqref{e-prime} that
\begin{equation}
\label{e1e1d}
\frac{(1-x)^2}{\epsilon^{(n+1)/2}(x)} \leq \epsilon^{\frac{n+1}6}(1)\left(\frac{\epsilon(1)}{\epsilon(1-\delta_\alpha^{n+1})}\right)^{\frac{n+1}2} = (1+o(1))\epsilon^{\frac{n+1}6}(1) = o_N\left( (n+1)^{-N} \right)
\end{equation}
for \( 1-\delta_\alpha^{n+1}\leq x \leq 1 \). Since
\begin{equation}
\label{Yn1}
Y_{n+1}(x) = (2\alpha/\rho^2)(x+1)R(x) + o_N\left( (n+1)^{-N}\right)
\end{equation}
on any fixed small enough neighborhood of \( 1 \), it holds that
\begin{equation}
\label{Tn1}
V_{n+1}(x) = -\frac2\rho\frac{h(x)R(x)}{1+x} + o_N\left( (n+1)^{-N}\right) = \frac{4\alpha}\rho + o_N\left( (n+1)^{-N}\right)
\end{equation}
uniformly for \( 1-\delta_\alpha^{n+1}\leq x \leq 1 \). Let
\[
Z_{n+1}(x) := \sqrt{Y_{n+1}(x)} - \frac{x-1}2\left( (n+1)\frac{1-x}{x}\frac{\sqrt{Y_{n+1}(x)}}{r(x)} + \frac{Y_{n+1}^\prime(x)}{\sqrt{Y_{n+1}(x)}}\right).
\]
It follows from the definition of \( Z_{n+1}(x) \) and an estimate similar to \eqref{e1e1d} that
\begin{multline*}
\int_{1-\delta_\alpha^{n+1}}^1\frac{\epsilon^{(n+1)/2}(x)Z_{n+1}(x)\,\dd x}{(1-x)^2+\epsilon^{n+1}(x)Y_{n+1}(x)} = \left.\arctan\left(\frac{x-1}{\sqrt{\epsilon^{n+1}(x)Y_{n+1}(x)}}\right)\right|_{1-\delta_\alpha^{n+1}}^1 \\
= \frac\pi2 - \arctan\left( \mathcal O(1)\epsilon^{\frac{n+1}6}(1)\right) = \frac\pi2 + o_N\left( (n+1)^{-N} \right).
\end{multline*}
Furthermore, we get from \eqref{Yn1}, the definition of \( Z_{n+1}(x) \), and \eqref{Tn1} that
\[
Z_{n+1}(x) = \frac{4\alpha}\rho + o_N\left( (n+1)^{-N}\right) = V_{n+1}(x) + o_N\left( (n+1)^{-N}\right)
\]
uniformly for \( 1-\delta_\alpha^{n+1}\leq x \leq 1 \). Therefore, \eqref{dens} yields that
\begin{eqnarray*}
\frac2\pi\int_{1-\delta_\alpha^{n+1}}^1 \frac{\sqrt{1-h_{n+1}^2(x)}}{1-x^2} \dd x &=& \frac2\pi\int_{1-\delta_\alpha^{n+1}}^1\frac{\epsilon^{(n+1)/2}(x)\big(Z_{n+1}(x) + o_N\left( (n+1)^{-N}\right)\big)}{(1-x)^2+\epsilon^{n+1}(x)Y_{n+1}(x)} \dd x \\
& = & 1 + o_N\left( (n+1)^{-N} \right),
\end{eqnarray*}
where we used positivity of the integrand for the last estimate.
\end{proof}

\begin{proof}[Proof of Theorem~\ref{thm:1}]
The claim follows from formula \eqref{Enalpha} and Lemmas~\ref{lem:8}--\ref{lem:10}.
\end{proof}

\end{document}